\documentclass{amsart}

\usepackage{amssymb}
\usepackage{bbm}

\usepackage{graphicx}
\usepackage[margin=3.9cm]{geometry}
\usepackage[utf8]{inputenc}

\newtheorem{thm}{Theorem}[section]
\newtheorem{theorem}[thm]{Theorem}
\newtheorem{lemma}[thm]{Lemma}

\theoremstyle{definition}
\newtheorem{definition}[thm]{Definition}
\newtheorem{example}[thm]{Example}

\theoremstyle{remark}

\newtheorem{remark}[thm]{Remark}

\numberwithin{equation}{section}

\begin{document}

\title{On the Count Probability of Many Correlated Symmetric Events}

\author{Rüdiger Kürsten}
\address{Institut für Physik, Universität Greifswald, Felix-Hausdorff-Str. 6, D-17489 Greifswald, Germany}
\email{ruediger.kuersten@uni-greifswald.de}
\date{October 29, 2019}

\begin{abstract}
We consider $N$ events that are defined on a common probability space.
Those events shell have a common probability function that is symmetric with respect to interchanging the events.
We ask for the probability distribution of the number of events that occur. 
If the probability of a single event is proportional to $1/N$ the resulting count probability is Poisson distributed in the limit of $N\rightarrow \infty$ for independent events.
In this paper we calculate the characteristic function of the limiting count probability distribution for events that are correlated up to an arbitrary but finite order.
\end{abstract}

\maketitle

\section{Introduction}

We might distribute $N$ grains of rice randomly in a room and ask for the number of grains that lie on a marked subset of the ground like e.g. a circle drawn on it.
If we assume that each grain lies on each position equally likely and furthermore that the positions of all grains are independent, the resulting number of grains in the marked area will be binomial distributed.
If we increase the amount of rice grains and the area of the room such that the ration between them remains constant, the binomial distribution will converge to a Poisson distribution in the limit $N \rightarrow \infty$ \cite{Poisson1837}.
In this paper we generalize this result for the case that the positions of the grains are not independent but still symmetric with respect to interchanges of the grains.
That means it is impossible to distinguish the different grains statistically.
We consider correlations of arbitrary but finite order. Later on we define precisely what is meant by correlation order.

Clearly, the probabilistic analysis of the problem is not restricted to rice grain but it can be applied to any large set of randomly positioned objects. The study of the distribution of large sets of points in space is known as spatial point pattern analysis \cite{Diggle83, IPSS08}.
It has applications in divers fields such as for example ecology \cite{VMGMW16}, astronomy \cite{KSRBBGMPW97} or statistics of crimes \cite{MSBST11}.
Furthermore, large sets of points appear as the positions of molecules in statistical physics.

For our purposes, the exact position of each point is not important because we only care if it lies within the marked area or not.
Hence we can consider the event that a given particle lies within the marked space.
In that way we obtain $N$ events from the positions of $N$ particles.
In statistical physics, those events are usually correlated due to interactions between the particles and also other objects are often correlated.
Formally, we consider $N$ events that are correlated and statistically indistinguishable.
In that formulation it is not important whether those events are related to positions of particles in space.

Our main result is an explicit formula for the characteristic function of the number of events that occur in the limit $N\rightarrow \infty$.
If correlations are limited to some order $l_{\text{max}}$, the characteristic function will depend on exactly $l_{\text{max}}$ parameters.
The formula was already given in \cite{KSZI19} for the particular application to spatial point distributions, however, without proof.
Ref. \cite{KSZI19} also gives an efficient Monte Carlo algorithm to sample the parameters of the distribution in the case of spatial point patterns.
Furthermore, \cite{KSZI19} investigates not only the number of particles within a marked set in space but also the number of neighbors of a randomly chosen particle.
If the neighborhood is defined by some spatial relation, e.g. if particles are considered to be neighbors if they lie within some given distance, the problem will be equivalent to the number of particles within an arbitrarily placed circle for independent homogeneously distributed particles.
However, for correlated particles the two problems are not identical, but related.
In fact the number of neighbor distribution can be obtained from our main result and depends on $l_{\text{max}}$ additional parameters, cf. \cite{KSZI19}.

The paper is organized as follows.
In Section \ref{sec:definitions} we give some basic definitions introducing for example the common probability function of $N$ events, their correlation function of order $k$, and the count probability function that gives the probability that exactly $s$ of the $N$ events occur.
Furthermore, we give some Lemma that are useful later on.
In Section \ref{sec:result} we give our main result, Theorem \ref{theorem:main}, together with its proof.
In Section \ref{sec:discussion} we conclude with a short discussion of the result.

\section{Basic Definitions and some Lemma\label{sec:definitions}}

Instead of explicitly writing the intersections, unions or complements of the considered events we describe all relevant events using indicator functions.
\begin{definition}[Indicator Function]
	Let $E$ be an event defined on a probability space $(\Omega, \Sigma, P)$. We define its \underline{indicator function} $\mathbbm{1}_E: \Omega \rightarrow \{0,1 \}$ by
	\begin{align}
		\mathbbm{1}_E(\omega) := \begin{cases} 1 \text{ if } \omega \in E \\ 0 \text{ if } \omega \notin E. \end{cases}
		\label{eq:indicator}
	\end{align}
\end{definition}

\begin{definition}[Symmetric Sequence of Events]\label{def:symmetricsequenceofevents}
	We call a finite sequence of events $E_1, \cdots, E_N$ that are defined on a common probability space \underline{symmetric}, if for each finite sequence $r_1, \cdots, r_N$, with $r_i\in \{0,1\}$ for $i \in \{1,\cdots, N\}$ it holds
	\begin{align}
		P(\mathbbm{1}_{E_1}=r_1, \cdots, \mathbbm{1}_{E_N}=r_N)=P(\mathbbm{1}_{E_1}=r_{\sigma(1)}, \cdots, \mathbbm{1}_{E_N}=r_{\sigma(N)})
	\end{align}
	for all permutations $\sigma$ of the elements $\{1, \cdots, N\}$.
\end{definition}
That means it is impossible to distinguish the events statistically.

\begin{definition}[Probability Function]\label{def:probabilityfunction}
	Given a symmetric sequence of events $E_1, \cdots, E_N$, we call the function $P_k: \{0,1\}^k \rightarrow [0,1]$ defined by
	\begin{align}
		P_{k}(r_1, r_2, \cdots, r_k):= \sum_{r_l \in \{0,1\} \text{ for }l \in \{k+1,\cdots, n\}}P(\mathbbm{1}_{E_1}=r_1, \cdots, \mathbbm{1}_{E_N}=r_N)
		\label{eq:probabilityfunction}
	\end{align}
	the \underline{probability function of order $k$}, where $1\le k \le N$ and $r_i \in \{0,1\}$.
\end{definition}

\begin{remark}[Symmetry of the Probability Function] \label{remark:symmetryoftheprobabilityfunction}
	The probability functions are also symmetric, that means
	\begin{align}
		P_k(r_1, \dots, r_k)=P_k(r_{\sigma(1)}, \dots, r_{\sigma(k)})
		\label{eq:symmetryoftheprobabilityfunction}
	\end{align}
	for all permutations $\sigma$ of the elements $\{1, \dots, k\}$ and for all $k\in \{1,\dots, N\}$, which follows immediately from the Definitions \ref{def:symmetricsequenceofevents} and \ref{def:probabilityfunction}.
\end{remark}

\begin{definition}[Correlation Function]\label{def:correlationfunction}
	Given a symmetric sequence of events $E_1, \cdots, E_N$, we define the \underline{correlation function of order one} $G_1:\{0,1\} \rightarrow [0,1]$ as $G_1:\equiv P_1$ and \underline{the correlation functions of order $k$} $G_k:\{0,1\}^k \rightarrow \mathbb{R}$ for $1<k\le N$ recursively by
	\begin{align}
		G_k(r_1, \cdots, r_k):= &P_k(r_1, \cdots, r_k) 
		\label{eq:correlationfunction}
		\\
		&- \sum_{\sigma}\sum_{l=1}^{k-1} \frac{1}{(l-1)!} \frac{1}{(k-l)!}G_{l}(r_1, r_{\sigma(2)}, \cdots, r_{\sigma(l)}) P_{k-l}(r_{\sigma(l+1)}, \cdots, r_{\sigma(k)}),
		\notag
	\end{align}
	where the sum over $\sigma$ denotes the sum over all permutations of the elements $\{2, \cdots, k\}$.
\end{definition}
The idea of a decomposition in different correlation orders is due to Ursell who essentially introduced the expansion given by Definition \ref{def:correlationfunction}, however, not properly normalized to serve as a probability \cite{Ursell27}. Mayer and Montroll used exactly the expansion of Definition \ref{def:correlationfunction} \cite{MM41}. 
\begin{remark}[Alternative Formulation of the Definition of the Correlation Functions]\label{remark:correlationfunction}
	We can rewrite Eq.~\eqref{eq:correlationfunction} of Definition \ref{def:correlationfunction} as $P_l(\dots) = G_l(\dots) + \dots$ and insert it recursively for all $P_l$ of order $l<k$ into Eq.~\eqref{eq:correlationfunction} and eventually replace $P_1$ by $G_1$. Performing such an expansion, only $P_k$ and $G$-functions remain on the right hand side of Eq.~\eqref{eq:correlationfunction}.
	It follows inductively from Eq.~\eqref{eq:correlationfunction} that the indexes of all correlation functions on the right hand side are ordered, that is for each term $G_l(r_{i_1}, \dots, r_{i_l})$ appearing in the expansion of Eq.~\eqref{eq:correlationfunction} it holds $i_1<i_2<\dots<i_l$.
	Thus, instead of Eq.~\eqref{eq:correlationfunction} we might alternatively write
	\begin{align}
		&G_k(r_1, \dots, r_k)=P_k(r_1, \dots, r_k) - \sum \{ \text{over all possible products of $G$-functions such that}
		\notag
		\\
		&\text{each of the arguments $r_1, \dots, r_k$ appears exactly once and for each $G$-function}
		\notag
		\\
		&\text{the arguments are ordered.}\}
		\label{eq:alternativedefinitioncorrelationfunction}
	\end{align}
\end{remark}

\begin{example}[Three- and Four-Event Correlation Functions]
	Two examples of the alternative formulation of Definition \ref{def:correlationfunction} given in the Remark \ref{remark:correlationfunction} are the three- and four-event correlation functions $G_3$ and $G_4$ that are defined by
	\begin{align}
	G_3(r_1, r_2, r_3)=&P_3(r_1, r_2, r_3)-G_1(r_1)G_1(r_2)G_1(r_3)-G_1(r_1)G_2(r_2, r_3)
	\notag
	\\
	&-G_1(r_2)G_2(r_1, r_3)-G_1(r_3)G_2(r_1, r_2),
		\label{eq:g3}
		\\
	G_4(r_1, r_2, r_3, r_4)=&P_4(r_1, r_2, r_3, r_4)-G_1(r_1)G_1(r_2)G_1(r_3)G_1(r_4)
	\nonumber
	\\
	&-G_2(r_1,r_2)G_1(r_3)G_1(r_4)-G_2(r_1,r_3)G_1(r_2)G_1(r_4)
	\nonumber
	\\
	&-G_2(r_1,r_4)G_1(r_2)G_1(r_3)-G_2(r_2,r_3)G_1(r_1)G_1(r_4)
	\nonumber
	\\
	&-G_2(r_2,r_4)G_1(r_1)G_1(r_3)-G_2(r_3,r_4)G_1(r_1)G_1(r_2)
	\nonumber
	\\
	&-G_2(r_1,r_2)G_2(r_3, r_4)-G_2(r_1,r_3)G_2(r_2, r_4)-G_2(r_1,r_4)G_2(r_2, r_3)
	\nonumber
	\\
	&-G_3(r_1, r_2, r_3)G_1(r_4)-G_3(r_1, r_2, r_4)G_1(r_3)
	\nonumber
	\\
	&-G_3(r_1, r_3, r_4)G_1(r_2)-G_3(r_2, r_3, r_4)G_1(r_1).
	\end{align}
\end{example}

\begin{lemma}[Symmetry of the Correlation Function]\label{lemma:symmetryofthecorrelationfunction}
	Let $E_1, \dots, E_N$ be a symmetric sequence of correlated events, then the corresponding correlation functions are symmetric, that is
	\begin{align}
		G_k(r_1, \dots, r_k) = G_k(r_{\sigma(1)}, \dots, r_{\sigma(k)})
		\label{eq:symmetryofthecorrelationfunction}
	\end{align}
	for all permutations $\sigma$ of the elements $\{1, \dots, k\}$, where $1\le k \le N$.
\end{lemma}

\begin{proof}[Proof of Lemma \ref{lemma:symmetryofthecorrelationfunction}]
	Obviously, the statement is satisfied for $k=1$ as the identity is the only permutation of one element.
	We assume that the statement is true for all $k$ that satisfy $1\le k < l$ and aim to show that it also holds for $l$.
	Considering the definition of $G_l$ given in the formulation of Eq.~\eqref{eq:alternativedefinitioncorrelationfunction} we realize that $P_l$ on the right hand side is symmetric due to Remark \ref{remark:symmetryoftheprobabilityfunction}. 
	Thus it remains to show that the remaining terms (without $P_l$) on the right hand side of Eq.~\eqref{eq:alternativedefinitioncorrelationfunction} are symmetric. 
	We notice that those terms only consist of products of correlation functions $G_k$ with $k<l$ such that we can apply the induction hypothesis to them.
	Therefore we identify the terms $G_k(r_{i_\sigma(1)},\dots, r_{i_{\sigma(k)}} )$ for all permutations $\sigma$ of the elements $\{1, \dots, k\}$, where $i_j$ are some indexes.
	
	Then, for any permutation $\sigma$ of the elements $\{1, \dots, l \}$, we can assign for each term in the sum in Eq.~\eqref{eq:alternativedefinitioncorrelationfunction} $G_{k_1}(r_{i_1^{k_1}}, \dots, r_{i_{k_1}^{k_1}}) G_{k_2}(r_{i_1^{k_2}}, \dots, r_{i_{k_2}^{k_2}}) \dots G_{k_s}(r_{i_1^{k_s}}, \dots, r_{i_{k_s}^{k_s}})$ the term $G_{k_1}(r_{\sigma(i_1^{k_1})}, \dots, r_{\sigma(i_{k_1}^{k_1})}) G_{k_2}(r_{\sigma(i_1^{k_2})}, \dots, r_{\sigma(i_{k_2}^{k_2})}) \dots G_{k_s}(r_{\sigma(i_1^{k_s})}, \dots, r_{\sigma(i_{k_s}^{k_s})})$ which is due to the above identification also a term in the sum in Eq.~\eqref{eq:alternativedefinitioncorrelationfunction}.
	Here, $i_j^k$ are some indexes.
	Hence the permutation $\sigma$ only exchanges terms in the sum in Eq.~\eqref{eq:alternativedefinitioncorrelationfunction} which does not affect the result of the sum.
\end{proof}

\begin{definition}[Correlation Coefficient]\label{def:correlationcoefficient}
	Given a symmetric sequence of events $E_1, \cdots, E_N$ with correlation function of order $k$, $G_k$, with $1\le k \le N$ we call $C_k\in \mathbb{R}$ defined by
	\begin{align}
		C_k:= N^k G_k(1, \cdots, 1)
		\label{eq:correlationcoefficient}
	\end{align}
	the \underline{correlation coefficient of order $k$} associated with the symmetric sequence of events.
\end{definition}

\begin{lemma}[Vanishing Sum over Correlation Functions] \label{lemma1}
	Given a symmetric sequence of correlated events it holds for any correlation function $G_k$ of order $k\ge 2$
	\begin{align}
		G_k(1, r_2, \dots, r_k)=-G_k(0, r_2, \dots, r_k)
		\label{eq:correlationlemma}
	\end{align}
	for all values of $r_i\in \{0,1\}$ for $i\in \{2, \dots, k\}$.
\end{lemma}
This result was already given by Mayer and Montroll \cite{MM41}.

\begin{proof}[Proof of Lemma \ref{lemma1}]
	By Definition \ref{def:correlationfunction} we have for the two-event correlation function $G_2(r_1, r_2) = P_2(r_1, r_2) - P_1(r_1) P_2(r_2)$. 
	Summing this equation over $r_1=1,0$ we obtain the result for $k=2$.
	For $k>2$ we obtain the result by induction immediately from the Definition \ref{def:correlationfunction} when summing over $r_1=1,0$.
	The first term on the right gives $P_{k-1}(r_2, \cdots, r_k)$ and the $l=1$-term from the sum gives $-P_{k-1}(r_2, \cdots, r_k)$.
	All other terms of the sum are zero by induction hypothesis.
\end{proof}

\begin{lemma}[Reduction of Correlation Function]\label{lemma:reduction_of_correalation_function}
	Given a symmetric sequence of events $E_1, \cdots, E_N$ with correlation coefficient of order $k$, $C_k$, with $2\le k\le N$ it holds for the corresponding correlation function
	\begin{align}
		G_k(r_1, \cdots, r_k) = N^{-k} C_k \prod_{l=1}^{^k} (-1)^{r_l+1}.
		\label{eq:propcorrelationfunction}
	\end{align}
\end{lemma}

\begin{proof}[Proof of Lemma \ref{lemma:reduction_of_correalation_function}]
	By Lemma \ref{lemma1} the 'flipping' of $r_1$ in $G_k(r_1, r_2, \cdots, r_k)$ changes the value of the function by a factor of $-1$.
	Due to the symmetry given by Lemma \ref{lemma:symmetryofthecorrelationfunction} we obtain a factor of $-1$ also when flipping one of the other arguments of $G_k$.
	Thus, the claim follows directly from the Definition \ref{def:correlationcoefficient}.
\end{proof}

\begin{definition}[Correlation Free of Order $k$]
	We call a finite symmetric sequence of events $E_1, \cdots, E_N$ \underline{correlation free of order $k$} for $2\le k \le N$, if the corresponding correlation parameter is zero, that is
	\begin{align}
		C_k=0.
		\label{eq:correlation_free}
	\end{align}
\end{definition}

\begin{definition}[Correlated of Order $k$]
	We call a finite symmetric sequence of events $E_1, \cdots, E_N$ \underline{correlated of order $k$}, if it is not correlation free of order $k$, that is if
	\begin{align}
		C_k\neq 0.
		\label{eq:correlation_parameter}
	\end{align}
\end{definition}

\begin{definition}[Correlated up to Order $k$]
	We call a finite symmetric sequence of events $E_1, \cdots, E_N$ \underline{correlated up to order one}, if $C_k=0$ for all $k\ge 2$ and \underline{correlated up to order $k$} with $k\ge 2$ if $C_k\neq 0$ and $C_j=0$ for all $j>k$.
\end{definition}

\begin{definition}[Count Probability] \label{def:count_probability}
	Given a symmetric sequence of events $E_1, \cdots, E_N$, we call the function $p_N:\mathbb{N}_{0}\rightarrow [0,1]$ defined by
	\begin{align}
		p_{N}(s):= \langle \delta(s-\sum_{l=1}^{^N}\mathbbm{1}_{E_l}) \rangle,
		\label{eq:countprobability}
	\end{align}
	the \underline{count probability} of the sequence of symmetric events,
	where $\langle \cdot \rangle$ denotes the expectation value and
	\begin{align}
		\delta(k) = \begin{cases} 1 \text{ if } k=0 \\
			0 \text{ else.}\end{cases}
			\label{eq:delta}
		\end{align}
\end{definition}

\section{Main Result\label{sec:result}}

 \begin{theorem}[Poisson-like Distribution]\label{theorem:main}
	 Let $(E_1^N, E_2^N, \cdots, E_N^N)_{N \in \mathbb{N}, N \ge N_0}$ be a sequence of symmetric sequences of events such that for all $N\ge N_0$ the finite symmetric sequences of events $E^N_1, \cdots, E^N_N$ are correlated up to order $l_{\text{max}}$ with the same correlation parameters $C_1, \cdots, C_{l_{\text{max}}}$.
	 Let furthermore $p_N(s)$ be the count probability of the symmetric sequence of events $E^N_1, \cdots, E^N_N$, then
	 \begin{align}
		 \lim_{N\rightarrow \infty} p_N(s) = p_{\infty}(s), 
		 \label{eq:theorem}
	 \end{align}
	 where the characteristic function of the limiting distribution $p_{\infty}$ is given by
	 \begin{align}
		 \chi(u)= \sum_{s=0}^\infty \exp(ius)p(s) = \exp\bigg[ \sum_{l=1}^{l_{\text{max}}} \sum_{t=0}^l (-1)^{l-t} \frac{C_l}{l!} \binom{l}{t} \exp(itu) \bigg].
		 \label{eq:characteristicfunction}
	 \end{align}
 \end{theorem}

 \begin{proof}[Proof of Theorem \ref{theorem:main}]
 We start with Definition \ref{def:count_probability} which can be written as
 \begin{eqnarray}
	 p_{N}(s)= \sum_{r_i\in\{0,1\} } P_{N}(r_1, \cdots, r_N) \delta\bigg(s-\sum_{i=1}^N r_i \bigg).
	 \label{eq:definition_count_probability2}
 \end{eqnarray}
 For simplicity we first investigate $p_N(N)$.
 In that case, only the term with all $r_i=1$ in Eq.~(\ref{eq:definition_count_probability2}) contributes to the sum, as only for this term the $\delta$-function has the value one.
 Next, we replace the probability function $P_N$ by correlation functions according to Definition \ref{def:correlationfunction}. In particular we insert Eq.~(\ref{eq:alternativedefinitioncorrelationfunction}) of remark \ref{remark:correlationfunction}.
 In the resulting sum, all summands containing correlation functions of order greater than $l_{\text{max}}$ are zero because we assumed that those correlation functions are zero.
 From the remaining summands, many are equal due to the symmetry given by Lemma \ref{lemma:symmetryofthecorrelationfunction}.
 Each summand can be characterized by the numbers $k_l$ of factors of $G_l(1, 1, \cdots, 1)$ for $l=1, 2, \cdots, N$.
 The number of equal summands, that is the number of summands characterized by the same set of numbers $k_l$ is given by the combinatorial factor
 \begin{eqnarray}
	 \prod_{l=1}^N \bigg( \frac{1}{l!}\bigg)^{k_l} M^{n_l}_{l, k_l},
	 \label{eq:combinatorialfactor1}
 \end{eqnarray}
 where
 \begin{eqnarray}
	 n_l= N - \sum_{l'=1}^{l-1} k_{l'}.
	 \label{eq:combinatorialfactor2}
 \end{eqnarray}
The above sum is meant to have no terms if the upper bound is smaller than the lower.
The combinatorial factor $M^{n_l}_{l, k_l}$ denotes the number of possibilities to choose $l$ ordered $k_l$-plets from $n_l$ elements.
It can be calculated by 
\begin{eqnarray}
	M^{n_l}_{l, k_l}= \binom{n_l}{l k_l} \binom{l k_l}{k_l} \binom{l(k_l-1)}{k_l} \cdots \binom{k_l}{k_l} (k_l!)^l \frac{1}{k_l!}.
	 \label{eq:combinatorialfactor3}
\end{eqnarray}
The first factor gives the number of possibilities to choose the $l k_l$ arguments of the $l$ $k_l$-plets. The second factor gives the number of choices of the argument of the first $k_l$-plet, the next factor the number of possible choices for the second $k_l$-plet and so on.
The factor $(k_l!)^l$ gives the number of possible orders of the arguments of the $l$ $k_l$-plets and the last factor takes care of the fact that the $l$ $k_l$-plets are not ordered.
The above formula, Eq.~(\ref{eq:combinatorialfactor3}), simplifies to
\begin{eqnarray}
	M^{n_l}_{l, k_l}= \frac{n_l!}{(n_l-l k_l)!} \frac{1}{k_l!}.
	 \label{eq:combinatorialfactor4}
\end{eqnarray}
\begin{figure}
	\includegraphics[width=0.9\textwidth]{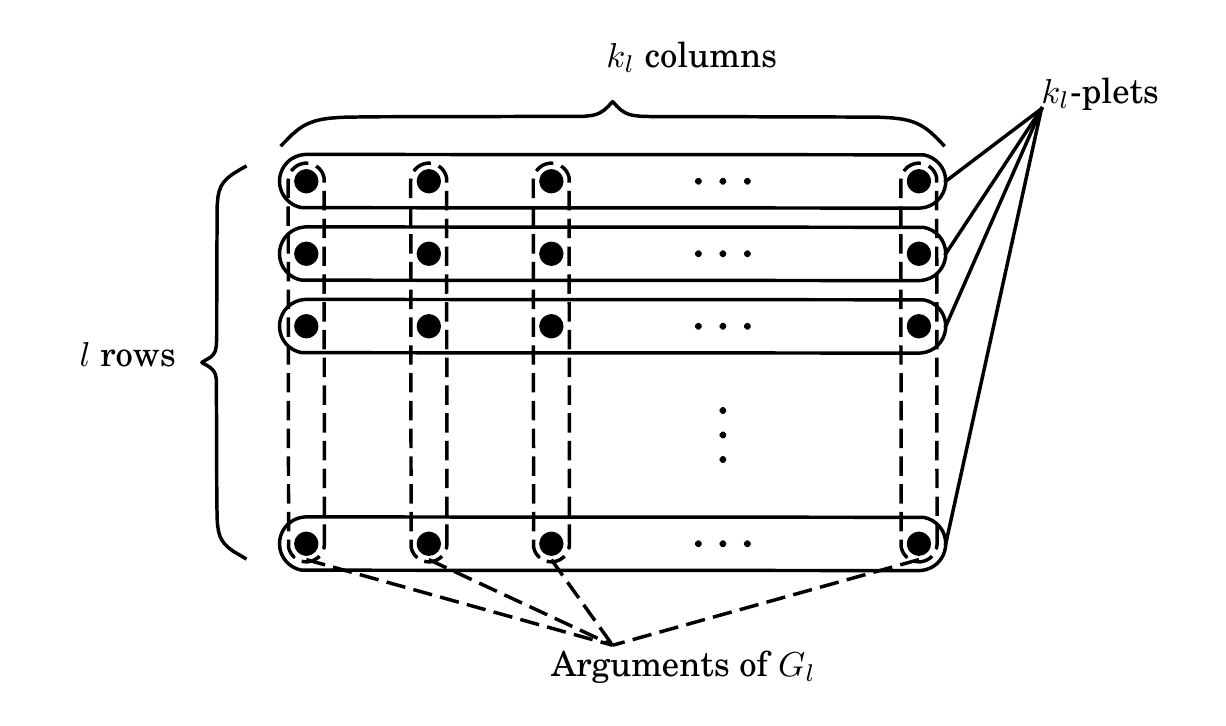}
	\caption{Illustration of the combinatorial factor $\bigg(\frac{1}{l!}\bigg)^{k_l}M^{n_l}_{l, k_l}$. \label{fig:1}}
\end{figure}
The construction of the combinatorial factors $\bigg( \frac{1}{l!}\bigg)^{k_l} M^{n_l}_{l, k_l}$ in Eq.~(\ref{eq:combinatorialfactor1}) is illustrated in Figure~\ref{fig:1}. 
There are $M^{n_l}_{l, k_l}$ possibilities to choose $l$ ordered $k_l$-plets. Than the arguments for each factor $G_l$ are chosen as the columns displayed in Figure~\ref{fig:1}.
In that way we obtain all possible permutations of the arguments for each $G_l$-factor.
However, in Eq.~(\ref{eq:alternativedefinitioncorrelationfunction}) the arguments of all correlation factors are ordered.
This is compensated by the factor of $\bigg( \frac{1}{l!}\bigg)^{k_l}$ in Eq.~(\ref{eq:combinatorialfactor1}).

Putting all together and collecting all the equal summands in $P_{N}(N)$ we obtain
\begin{eqnarray}
	p_N(N)=\prod_{l=1}^{l_{\text{max}}} \Bigg[ \sum_{k_l=0}^\infty \bigg(\frac{C_l}{N} \bigg)^{k_l} \bigg( \frac{1}{l!}\bigg)^{k_l} M^{n_l}_{l, k_l}\Bigg]\delta\bigg( N - \sum_{l=1}^{l_{\text{max}}}k_l \bigg), 
	\label{eq:pnn}
\end{eqnarray}
where we inserted $C_l/N$ for $G_l(1, 1, \cdots, 1)$ according to Definition \ref{def:correlationcoefficient}. The delta-function takes care of the fact that the total number of arguments of the correlation functions in each summand equals $N$.

For general arguments $s$ of $p_N(s)$ we go along the same lines.
We insert $P_N$ from Eq.~(\ref{eq:alternativedefinitioncorrelationfunction}) into Eq.~(\ref{eq:definition_count_probability2}) and collect equal summands from the sum over the values of the arguments $r_i$ and over the different products of correlation functions.
In analogy to Eq.~(\ref{eq:pnn}) we find
\begin{eqnarray}
	p_{N}(s) =\sum_{k_{1,0}}^{\infty} \Bigg[ \bigg(\frac{C_1}{N} \bigg)^{k_{1,0}} M^{n_{1,0}}_{1, k_{1,0}} \Bigg] \times \sum_{k_{1,1}}^{\infty} \Bigg[ \bigg(1- \frac{C_1}{N} \bigg)^{k_{1,1}} M^{n_{1,1}}_{1, k_{1,1}} \Bigg]
	\nonumber
	\\
	\times \prod_{l=2}^{l_{\text{max}}} \Bigg\{ \prod_{q=0}^{l} \Bigg[ \sum_{k_{l,q}=0}^{\infty} \bigg( (-1)^q \frac{C_l}{N^l} \bigg)^{k_{l,q}} M^{n_{l,q}}_{l, k_{l,q}} \bigg( \frac{1}{q!} \frac{1}{(l-q)!} \bigg)^{k_{l,q}}\Bigg] \Bigg\}
	\nonumber
	\\
	\times \delta\bigg(s-\sum_{l=1}^{l_{\text{max}}}\sum_{q=0}^l (l-q)k_{l,q} \bigg) \delta\bigg(N -  \sum_{l=1}^{l_{\text{max}}}\sum_{q=0}^l l k_{l,q} \bigg), 
	\label{eq:pns1}
\end{eqnarray}
where
\begin{eqnarray}
	n_{l,q} = N - \sum_{l'=1}^{l-1}\sum_{q'=0}^{l'} k_{l', q'} - \sum_{q'=0}^{q-1}k_{l, q'}.
	\label{eq:nlq}
\end{eqnarray}
Here, $k_{l, q}$ denotes the number of factors $G_l$ with $q$ arguments that have the value zero and $l-q$ arguments that have the value one.
That is $l$ denotes the correlation order and $q$ denotes the number of arguments that are zero.
The factors $\frac{C_1}{N}$ and $1-\frac{C_1}{N}$ are, according to Definition \ref{def:correlationcoefficient}, the values of the correlation function $G(1)$ and $G(0)$, respectively.
The factor $(-1)^q\frac{C_l}{N^l}$ gives the value of $G_l$ with $q$ arguments with the value zero and $l-q$ arguments with the value one according to Lemma \ref{lemma:reduction_of_correalation_function}. 
The factor $\frac{1}{q!} \frac{1}{(l-q)!}$ is the analog of $\frac{1}{l!}$ in Eq.~(\ref{eq:pnn}). However, here we do not divide by the number of possible permutations of the arguments of $G_l$ but by the number of possible permutation of only that arguments that have the value one and only that arguments that have the value zero.
The first delta-function ensures that the total number of arguments of the correlation functions in each summand that have the value one equals $s$ and the second delta-function ensures that the total number of arguments of the correlation functions in each summand equals $N$.

Inserting the combinatorial factor $M^{n_{l,q}}_{l, k_{l,q}}$ given by Eq.~(\ref{eq:combinatorialfactor4}) into Eq.~(\ref{eq:pns1}) the product of the first factor in Eq.~(\ref{eq:combinatorialfactor4}) results in a factor $N!$ as most of the terms cancel.
Eventually we obtain from Eq.~(\ref{eq:combinatorialfactor4}) 
\begin{eqnarray}
	p_{N}(s) =N! \sum_{k_{1,0}}^{\infty} \Bigg[ \bigg(\frac{C_1}{N} \bigg)^{k_{1,0}} \frac{1}{k_{1,0}!} \Bigg] \times \sum_{k_{1,1}}^{\infty} \Bigg[ \bigg(1- \frac{C_1}{N} \bigg)^{k_{1,1}} \frac{1}{k_{1,1}!} \Bigg]
	\nonumber
	\\
	\times \prod_{l=2}^{l_{\text{max}}} \Bigg\{ \prod_{q=0}^{l} \Bigg[ \sum_{k_{l,q}=0}^{\infty} \bigg( (-1)^q \frac{C_l}{N^l} \bigg)^{k_{l,q}} \frac{1}{k_{l, q}!} \bigg( \frac{1}{q!} \frac{1}{(l-q)!} \bigg)^{k_{l,q}}\Bigg] \Bigg\}
	\nonumber
	\\
	\times \delta\bigg(s-\sum_{l=1}^{l_{\text{max}}}\sum_{q=0}^l (l-q)k_{l,q} \bigg) \delta\bigg(N -  \sum_{l=1}^{l_{\text{max}}}\sum_{q=0}^l l k_{l,q} \bigg). 
	\label{eq:pns2}
\end{eqnarray}
Evaluating the sum over $k_{1,1}$ taking into account the second delta-function we obtain
\begin{eqnarray}
	p_{N}(s) =\frac{N!}{(N-\sum_{l=2}^{l_{\text{max}}}\sum_{q=0}^l l k_{l,q}  - k_{1,0})!} \sum_{k_{1,0}}^{\infty} \Bigg[ \bigg(\frac{C_1}{N} \bigg)^{k_{1,0}} \frac{1}{k_{1,0}!} \Bigg] 
	\nonumber
	\\
	\times  \bigg(1- \frac{C_1}{N} \bigg)^{N-\sum_{l=2}^{l_{\text{max}}}\sum_{q=0}^l l k_{l,q} -k_{1,0}} 
	\nonumber
	\\
	\times \prod_{l=2}^{l_{\text{max}}} \Bigg\{ \prod_{q=0}^{l} \Bigg[ \sum_{k_{l,q}=0}^{\infty} \bigg( (-1)^q \frac{C_l}{N^l} \bigg)^{k_{l,q}} \frac{1}{k_{l, q}!} \bigg( \frac{1}{q!} \frac{1}{(l-q)!} \bigg)^{k_{l,q}}\Bigg] \Bigg\}
	\nonumber
	\\
	\times \delta\bigg(s-\sum_{l=1}^{l_{\text{max}}}\sum_{q=0}^l (l-q)k_{l,q} \bigg).
	\label{eq:pns3}
\end{eqnarray}
The first factor can be written as
\begin{eqnarray}
	\frac{N!}{(N-\sum_{l=2}^{l_{\text{max}}}\sum_{q=0}^l l k_{l,q}  - k_{1,0})!} = N^{\sum_{l=2}^{l_{\text{max}}}\sum_{q=0}^l l k_{l,q}  + k_{1,0}}[1+ R(N)],
	\label{eq:largeN}
\end{eqnarray}
where the remainder $R(N)$ goes to zero as $\propto 1/N$ for $N\rightarrow \infty$.
Inserting Eq.~(\ref{eq:largeN}) into Eq.~(\ref{eq:pns3}) almost all powers of $N$ cancel and we obtain
\begin{eqnarray}
	p_{N}(s) = [1+R(N)]\sum_{k_{1,0}}^{\infty} \Bigg[  \frac{C_1^{k_{1,0}}}{k_{1,0}!} \Bigg] 
	\bigg(1- \frac{C_1}{N} \bigg)^{N-\sum_{l=2}^{l_{\text{max}}}\sum_{q=0}^l l k_{l,q} -k_{1,0}} 
	\nonumber
	\\
	\times \prod_{l=2}^{l_{\text{max}}} \Bigg\{ \prod_{q=0}^{l} \Bigg[ \sum_{k_{l,q}=0}^{\infty} \bigg( (-1)^q C_l \bigg)^{k_{l,q}} \frac{1}{k_{l, q}!} \bigg( \frac{1}{q!} \frac{1}{(l-q)!} \bigg)^{k_{l,q}}\Bigg] \Bigg\}
	\nonumber
	\\
	\times \delta\bigg(s-\sum_{l=1}^{l_{\text{max}}}\sum_{q=0}^l (l-q)k_{l,q} \bigg).
	\label{eq:pns4}
\end{eqnarray}
Performing the limit $N\rightarrow \infty$ yields
\begin{eqnarray}
	p_{\infty}(s) = \sum_{k_{1,0}}^{\infty} \Bigg[  \frac{C_1^{k_{1,0}}}{k_{1,0}!} \Bigg] \exp(-C_1)
	\nonumber
	\\
	\times \prod_{l=2}^{l_{max}} \Bigg\{ \prod_{q=0}^{l} \Bigg[ \sum_{k_{l,q}=0}^{\infty} \bigg( (-1)^q C_l \bigg)^{k_{l,q}} \frac{1}{k_{l, q}!} \bigg( \frac{1}{q!} \frac{1}{(l-q)!} \bigg)^{k_{l,q}}\Bigg] \Bigg\}
	\nonumber
	\\
	\times \delta\bigg(s-\sum_{l=1}^{l_{\text{max}}}\sum_{q=0}^l (l-q)k_{l,q} \bigg).
	\label{eq:pns5}
\end{eqnarray}
Now, we perform the sum over $k_{1,0}$ taking the delta-function into account to obtain
\begin{eqnarray}
	p_{\infty}(s) = C_1^s \exp(-C_1)\frac{1}{(s-\sum_{l=2}^{l_{\text{max}}}\sum_{q=0}^l (l-q)k_{l,q} )!}
	\nonumber
	\\
	\times \prod_{l=2}^{l_{\text{max}}} \Bigg\{ \prod_{q=0}^{l} \Bigg[ \sum_{k_{l,q}=0}^{\infty} \bigg(\frac{(-1)^q C_l (C_1)^{q-l}}{q!(l-q)!}  \bigg)^{k_{l,q}} \frac{1}{k_{l, q}!}  \Bigg] \Bigg\}.
	\label{eq:pns6}
\end{eqnarray}
The first line is independent of $k_{l,l}$. Hence we can perform the sums over $k_{l,l}$ that result in exponential factors yielding
\begin{eqnarray}
	p_{\infty}(s) = C_1^s \exp\bigg(\sum_{l=1}^{l_{\text{max}}} (-1)^l \frac{C_l}{l!}\bigg)\frac{1}{(s-\sum_{l=2}^{l_{\text{max}}}\sum_{q=0}^{l-1} (l-q)k_{l,q} )!}
	\nonumber
	\\
	\times \prod_{l=2}^{l_{\text{max}}} \Bigg\{ \prod_{q=0}^{l-1} \Bigg[ \sum_{k_{l,q}=0}^{\infty} \bigg(\frac{(-1)^q C_l (C_1)^{q-l}}{q!(l-q)!}  \bigg)^{k_{l,q}} \frac{1}{k_{l, q}!}  \Bigg] \Bigg\}.
	\label{eq:pns7}
\end{eqnarray}
Next, we substitute $q$ by $t=l-q$ and calculate the characteristic function resulting in
\begin{eqnarray}
	\chi(u)=\sum_{s=0}^\infty p_{\infty}(s)\exp(ius) 
	\nonumber
	\\
	= \sum_{s=0}^\infty \exp(ius) C_1^s \exp\bigg(\sum_{l=1}^{l_{\text{max}}} (-1)^l \frac{C_l}{l!}\bigg)\frac{1}{(s-\sum_{l=2}^{l_{\text{max}}}\sum_{t=1}^{l} t k_{l,t}  )!}
	\nonumber
	\\
	\times \prod_{l=2}^{l_{\text{max}}} \Bigg\{ \prod_{t=1}^{l} \Bigg[ \sum_{k_{l,t}=0}^{\infty} \bigg(\frac{(-1)^{l-t} C_l (C_1)^{-t}}{t!(l-t)!}  \bigg)^{k_{l,t}} \frac{1}{k_{l, t}!}  \Bigg] \Bigg\}.
	\label{eq:characteristic_function}
\end{eqnarray}
We can evaluate the sum over $s$ using
\begin{eqnarray}
	\sum_{s=0}^\infty \frac{\bigg[C_1 \exp(iu)\bigg]^s}{(s-a)!}=\sum_{s=a}^\infty \frac{\bigg(C_1 \exp(iu)\bigg)^s}{(s-a)!}
	\nonumber
	\\
	=[C_1 \exp(iu)]^a \sum_{\tilde{s}=0}^\infty \frac{ \bigg(C_1 \exp(iu)\bigg) ^{\tilde{s}} }{\tilde{s}!}=[C_1 \exp(iu)]^a\exp\Bigg[C_1 \exp(iu) \Bigg], 
	\label{eq:sumovers}
\end{eqnarray}
where we used that $1/(s-a)!$ is considered to be zero if $s-a<0$ in the first equation.
Furthermore we used the substitution $\tilde{s}=s-a$. 
With $a=\sum_{l=2}^{l_{\text{max}}}\sum_{t=1}^{l} t k_{l,t}$ we obtain from Eq.~(\ref{eq:characteristic_function}) with Eq.~(\ref{eq:sumovers})
\begin{eqnarray}
	\chi(u)= \exp\Bigg[C_1 \exp(iu) \Bigg] \exp\bigg(\sum_{l=1}^{l_{\text{max}}} (-1)^l \frac{C_l}{l!}\bigg)
	\nonumber
	\\
	\times \prod_{l=2}^{l_{\text{max}}} \Bigg\{ \prod_{t=1}^{l} \Bigg[ \sum_{k_{l,t}=0}^{\infty} \bigg(\frac{(-1)^{l-t} C_l \exp(iul) }{t!(l-t)!}  \bigg)^{k_{l,t}} \frac{1}{k_{l, t}!}  \Bigg] \Bigg\}.
	\label{eq:characteristic_function2}
\end{eqnarray}
Eventually we can evaluate the remaining sums over $k_{l,t}$ which result in exponential functions. Collecting all terms we obtain Eq.~(\ref{eq:characteristicfunction}).
 \end{proof}
 It is remarkable that one could also evaluate the sums over $k_{l, l-1}$ directly in Eq.~(\ref{eq:pns7}).
 However, evaluating the remaining sums directly seems to be not doable.

 \section{Summary and Discussion\label{sec:discussion}}
 We calculated the probability distribution of the number of occurring events from a set of $N$ correlated events in the limit $N \rightarrow \infty$ under the assumption that the events are statistically indistinguishable and correlations are limited to an arbitrary order $l_{\text{max}}$.
 The limiting distribution is is given by Eq.~(\ref{eq:pns7}) that can be further simplified evaluating the sums over $k_{l,l-1}$, however, this expression contains still infinite sums. 
 Therefore it is not suitable for practical evaluations.
 We calculate the characteristic function of the limiting distribution, Eq.~(\ref{eq:characteristicfunction}), which has a surprisingly simple form containing only finite sums.
 Setting all correlation parameters $C_k=0$ for $k>1$ we recover the characteristic function of the Poisson distribution.




\end{document}